\documentclass[xcolor=dvipsnames,svgnames,table,reqno,11pt]{amsart}

\usepackage[left=1.5in, right=1.5in, bottom=1.25in, height=8in]
{geometry}

\input xy
\xyoption{all}
\usepackage{accents}
\usepackage{epsfig}
\usepackage{xcolor}
\usepackage{amsthm}
\usepackage{amssymb}
\usepackage{bm} 
\usepackage{amsmath}
\usepackage{amscd}
\usepackage{amsopn}
\usepackage{graphicx}

\usepackage{tikz}
\usetikzlibrary{shapes,shadows,arrows}
\usepackage{tikz-cd}

\usepackage{xspace}

\usepackage{hhline}
\usepackage{easybmat}
\usepackage{caption}   
\usepackage{relsize}

\usepackage{url}
\usepackage{enumitem, hyperref}\hypersetup{colorlinks}

\usepackage{stmaryrd}

\usepackage[T1]{fontenc}

\colorlet{purpleB70}{blue!70!red}

\colorlet{orangeR65}{red!65!yellow}

\definecolor{red2}{HTML}{d41173}

\definecolor{neongreen}{HTML}{1bf702}

\definecolor{radicalred}{HTML}{FF355E}

\definecolor{denim}{HTML}{1560BD}

\definecolor{darkcyan}{rgb}{0.0, 0.55, 0.55}

\definecolor{cilek}{HTML}{FF43A4}

\definecolor{mor}{HTML}{9F00C5}

\definecolor{phlox}{rgb}{0.87, 0.0, 1.0}

\definecolor{fluorescentpink}{HTML}{FF1493}

\definecolor{napiergreen}{rgb}{0.16, 0.5, 0.0}

\definecolor{kellygreen}{rgb}{0.3, 0.73, 0.09}

\definecolor{parisgreen}{HTML}{ 50C878 }

\definecolor{palatinateblue}{rgb}{0.15, 0.23, 0.89}

\definecolor{ceruleanblue}{rgb}{0.16, 0.32, 0.75}

\definecolor{brandeisblue}{rgb}{0.0, 0.44, 1.0}

\definecolor{KLMblue}{HTML}{0FC0FC}

\definecolor{cinnamon}{rgb}{0.82, 0.41, 0.12}

\definecolor{darkorange}{rgb}{1.0, 0.55, 0.0}

\definecolor{darktangerine}{rgb}{1.0, 0.66, 0.07}

\definecolor{deepcarrotorange}{rgb}{0.91, 0.41, 0.17}

\definecolor{internationalorange}{HTML}{FF4F00}

\definecolor{persimmon}{HTML}{EC5800}

\definecolor{pumpkin}{HTML}{FF7518}

\definecolor{darkred}{rgb}{1,0,0} 
\definecolor{darkgreen}{rgb}{0,0.7,0}
\definecolor{darkblue}{rgb}{0,0,1}

\hypersetup{colorlinks,
linkcolor=darkblue,
filecolor=darkgreen,
urlcolor=darkred,
citecolor=darkgreen}

\makeatletter
\def\reflb#1#2{\begingroup
    #2%
    \def\@currentlabel{#2}%
    \phantomsection\label{#1}\endgroup
}
\makeatother

\numberwithin{equation}{section}
\newtheorem{Theorem}{Theorem}
\numberwithin{Theorem}{section}

\newtheorem   {Lemma}[Theorem]{Lemma}

\newtheorem   {Proposition}[Theorem]{Proposition}
\newtheorem   {Corollary}[Theorem]{Corollary}
\theoremstyle {definition}

\theoremstyle {remark}
\newtheorem   {Remark}[Theorem]{Remark}
\newtheorem   {Example}[Theorem]{Example}

\def    \eps    {\epsilon}

\newcommand{\CO}{{\mathcal O}}

\newcommand{\CS}{{\mathcal S}}

\newcommand{\supp}{\operatorname{supp}}

\newcommand{\id}{{\mathit id}}

\newcommand{\fb}{{\mathfrak b}}

\newcommand{\hL}{\hat{L}}

\def    \R      {{\mathbb R}}

\def    \N      {{\mathbb N}}

\def    \T      {{\mathbb T}}

\def    \12     {{\frac{1}{2}}}

\def    \p      {\partial}

\def    \vol     {\operatorname{vol}}

\def    \hh    {\mathrm{h}}

\newcommand    \htop  {\operatorname{h_{\scriptscriptstyle{top}}}}
\newcommand    \hvol  {\operatorname{h_{\scriptscriptstyle{vol}}}}
\newcommand   \hbr {\operatorname{\hbar}}

\begin{document}


\setlength{\smallskipamount}{6pt}
\setlength{\medskipamount}{10pt}
\setlength{\bigskipamount}{16pt}





\title [From Barcode Entropy to Metric Entropy]{From Barcode Entropy to Metric Entropy}

\author[Erman \c C\. inel\. i]{Erman \c C\. inel\. i}
\author[Viktor Ginzburg]{Viktor L. Ginzburg}
\author[Ba\c sak G\"urel]{Ba\c sak Z. G\"urel}

\address{E\c C: ETH Z\"urich, R\"amistrasse 101, 8092 Z\"urich,
  Switzerland}
\email{erman.cineli@math.eth.ch}

\address{VG: Department of Mathematics, UC Santa Cruz, Santa Cruz, CA
  95064, USA} \email{ginzburg@ucsc.edu}

\address{BG: Department of Mathematics, University of Central Florida,
  Orlando, FL 32816, USA} \email{basak.gurel@ucf.edu}



\date{\today} 

\thanks{The work is partially supported by the NSF grants DMS-2304207
  (BG) and DMS-2304206 (VG), the Simons Foundation grants 855299 (BG)
  and MP-TSM-00002529 (VG), and the ERC Starting Grant 851701 via a
  postdoctoral fellowship (E\c{C})}

\begin{abstract}
  We establish a connection between barcode entropy and metric
    entropy. Namely, we show
    that the barcode entropy bounds the metric entropy from below for
    a measure from a specific class of invariant measures associated
    with a pair of Lagrangian or Legendrian submanifolds, which we
    call \emph{pseudo-chord measures}. This inequality refines, via
    the variational principle, the previously known upper bound of
    barcode entropy by topological entropy.
 \end{abstract}

\maketitle



\tableofcontents

\section{Introduction}
\label{sec:intro}
In this paper we continue the study of barcode entropy and establish
its connection with metric entropy for a certain class of invariant
measures of Hamiltonian diffeomorphisms and Reeb flows.

Barcode entropy is a Floer theoretic counterpart of topological
entropy. For Hamiltonian diffeomorphisms, the barcode entropy is the
exponential growth rate of the number of not-so-short bars in the
filtered Floer homology as a function of iteration. For Reeb flows,
one looks at this number for the filtered symplectic homology as a
function of the action. These invariants come in two forms: absolute
and relative. The absolute version uses fixed point Floer homology or
symplectic homology and thus provides a lower bound on
the growth of periodic orbits. In the relative version, we have a pair
of Lagrangian (or Legendrian) submanifolds fixed and work with
filtered Lagrangian Floer homology or wrapped Floer homology. Thus,
the relative barcode entropy gives a lower bound on the exponential
growth rate of the number of chords. The latter lower bounds are more
robust under small perturbations of submanifolds than the growth of
the actual number of chords. This fact leads to a relation between
barcode entropy and the volume growth and ultimately topological
entropy.

Barcode entropy was originally introduced in \cite{CGG:Entropy} in the
Hamiltonian setting and then extensively studied in
\cite{Ci,CGG:Growth,CGGM:Entropy,FLS,Fe1,Fe2,GGM,Me}. We will review
the definitions and basic properties in Section \ref{sec:barcode}.

The feature of barcode entropy which is absolutely essential for what
follows is that it bounds topological entropy from below. (In some
cases, there is also an opposite inequality capturing topological
entropy of hyperbolic sets, but this fact is less important for our
purposes in this paper.) Here we refine this inequality by showing
that in the relative case the barcode entropy bounds from below the
metric entropy for a measure from a specific class of invariant
measures associated with a pair of Lagrangian or Legendrian
submanifolds. We refer to this class as the \emph{pseudo-chord
  measures} and it is essential that pseudo-chord measures depend on
the Lagrangian/Legendrian submanifolds. Then the original bound with
topological entropy simply follows from the fact that the metric
entropy is bounded from above by the topological entropy; see, e.g.,
\cite{KH}. The main new and non-obvious point here is the definition
of a pseudo-chord measure given in Sections \ref{sec:Lagr} and
\ref{sec:Reeb}.

A similar inequality holds for absolute barcode entropy with chord
measures replaced by what we call \emph{pseudo-periodic measures}; see
Section \ref{sec:Per}. However, this class of measures is quite large
and includes all ergodic measures. As a consequence, the inequality,
taken literally, does not give new information; for it follows from
the variational principle and the known relations between barcode
entropy and topological entropy; \cite{CGG:Entropy,FLS}.

\medskip
\noindent{\bf Acknowledgements} We are grateful to David Burguet
for explaining to us some details of Yomdin's theory.

\section{Main results}
\label{sec:main-results}
\subsection{Barcode entropy}
\label{sec:barcode}
In this section we briefly review the definition of barcode
entropy. We start with the case of Hamiltonian diffeomorphisms and
then recast the definitions in the framework of Reeb flows.  For
further details and references, we refer the reader to
\cite{CGG:Entropy, CGGM:Entropy}.

Let $L_0$ and $L$ be two closed, monotone, Hamiltonian isotopic
Lagrangian submanifolds in a symplectic manifold $M$. We require $M$
to be compact or sufficiently well-behaved at infinity (e.g., convex)
so that the filtered Floer homology for the pair $(L_0,L)$ is
defined. This homology is not a persistence module in the conventional
sense, but its barcode is still defined at least when $L_0$ and $L$
are transverse and each bar is taken up to a shift; see
\cite{UZ}. Denote by $\fb_\eps(L_0,L)$ the number of bars of length
greater than $\eps>0$. This definition extends to not-necessarily
transverse pairs essentially by continuity; \cite[Sec.\
4]{CGG:Entropy}.

Next, let $\varphi\colon M\to M$ be a compactly supported
$C^\infty$-smooth Hamiltonian diffeomorphism.  By definition, the
\emph{$\eps$-barcode entropy of $\varphi$ relative to $L_0$ and $L$}
is
\begin{equation}
  \label{eq:hbar-eps}
  \hbr_\eps(\varphi; L_0,L):=\limsup_{k\to \infty}\frac{\log^+
    b_\eps(L_k,L)}{k}\textrm{ where }  L_k:=\varphi^k(L_0).
\end{equation}
Here and throughout the paper, the logarithm is taken with the same
base as in the definition of topological entropy, and $\log^+=\log$
except that $\log^+0=0$. Then the \emph{barcode entropy of $\varphi$
  relative to $L_0$ and $L$} is
 \begin{equation}
  \label{eq:hbar}
  \hbr(\varphi;L_0,L):=\lim_{\eps\to 0+} \hbr_\eps (\varphi;
  L_0,L)\in [0, \infty].
\end{equation}

When $M$ is closed and monotone or weakly monotone, we can apply this
construction to the Hamiltonian diffeomorphism $\id\times \varphi$ of
the symplectic square $M\times M$ with $L_0=L$ taken to be the
diagonal $\Delta\subset M\times M$. As a result, we obtain the
(absolute) $\eps$-barcode entropy $\hbar_\eps(\varphi)$ of $\varphi$
and the \emph{(absolute) barcode entropy} $\hbar(\varphi)$. More
explicitly, $\hbar_\eps(\varphi)$ is the exponential growth rate,
\eqref{eq:hbar-eps}, of the number, $\fb_\eps(\varphi^k)$, of bars of
length greater than $\eps>0$ in the filtered Floer homology of
$\varphi^k$.

Barcode entropy is closely related to topological entropy. Namely,
\cite[Thm.\ A]{CGG:Entropy} asserts that
\begin{equation}
  \label{eq:hbar<htop-rel}
\hbar(\varphi; L_0,L)\leq \htop(\varphi)
\end{equation}
and, in particular, 
\begin{equation}
  \label{eq:hbar<htop}
\hbar(\varphi)\leq \htop(\varphi).
\end{equation}
As a consequence, barcode entropy is finite. (The requirement that
$\varphi$ is $C^\infty$-smooth is essential at least on the level of
proofs; for the argument relies on Yomdin's theorem; \cite{Yo}.)
Inequality \eqref{eq:hbar<htop} can be strict: there exists a
Hamiltonian diffeomorphism $\varphi$ of a closed symplectically
aspherical manifold such that $\hbar(\varphi)=0$ but
$\htop(\varphi)>0$; see \cite{Ci}.

Furthermore, when $K\subset M$ is a compact hyperbolic invariant set
of $\varphi$,
\begin{equation}
  \label{eq:htop-hyp<hbar}
\htop(\varphi|_K)\leq \hbar(\varphi)
\end{equation}
by \cite[Thm.\ B]{CGG:Entropy}. Therefore, combining these
inequalities with the results of Katok from \cite{Ka80}, we see that
$$
\hbar(\varphi)=\htop(\varphi)\textrm{ when } \dim M=2;
$$
\cite[Thm.\ C]{CGG:Entropy}. Inequality \eqref{eq:htop-hyp<hbar} is
extended to the relative case in \cite[Thm.\ 1]{Me}. Namely, assume
that $L_0$ and $L$ intersect $K$ and that $L_0$ contains a ball
centered at $p\in L_0\cap K$ in the unstable manifold of $K$ at $p$
and $L$ contains a ball centered at $q\in L\cap K$ in the stable
manifold. Furthermore, we require $K$ to be locally maximal and
topologically transitive. Then
  \begin{equation}
  \label{eq:htop-hyp<hbar-rel}
  \htop(\varphi|_K)\leq \hbar(\varphi;L_0,L).
\end{equation}
Formally speaking, this inequality does not directly imply
\eqref{eq:htop-hyp<hbar}, i.e., the two inequalities are logically
independent.
  
These constructions and results extend to Reeb flows. Consider the
Reeb flow $\varphi:=\{\varphi^t\mid t\in \R\}$ on the boundary $M$ of
a Liouville domain $W$. Then the filtered symplectic homology of $W$
is a genuine persistence module (see, e.g., \cite{CGGM:Entropy,FLS}),
and we denote by $\fb_\eps(s;\varphi)$ the number of bars of length
greater than $\eps>0$ beginning below $s\in \R$. (Here the grading of
the homology is inessential and we need not make any assumptions on
$c_1(TW)$.) Then the $\eps$-barcode entropy and the barcode entropy of
$\varphi$ is defined similarly to \eqref{eq:hbar-eps} and
\eqref{eq:hbar}:
  \begin{equation}
  \label{eq:hbar-Reeb}
  \hbr_\eps(\varphi):=\limsup_{s\to \infty}\frac{\log^+
    \fb_\eps(s;\varphi)}{s}\textrm{ and }  
  \hbr(\varphi):=\lim_{\eps\to 0+} \hbr_\eps (\varphi).
\end{equation}

In the relative version, following \cite{Fe1}, we start with closed
Legendrian submanifolds $L_0$ and $L$ in $M=\p W$ and require
$L_0$ and $L$ to admit exact compact Lagrangian fillings $\hL_0$
and $\hL$ to $W$ such that $L_0=\p \hL_0$ and $L=\p \hL$. Then we can
also make $\hL_0$ and $\hL$ be asymptotically conical Lagrangian
submanifolds in $W$. The latter condition means that, say, $\hL$ has
the form $\hL=L\times (1-\eps,1]$ within a symplectic collar
$M\times (1-\eps, 1]$ of $M=\p W$ in $W$.

Then in place of symplectic homology we use the wrapped Floer homology
persistence module for $(\varphi, \hL_0,\hL)$ to define
$\fb_\eps(s; L_0,L)$, the number of bars of length greater than $\eps$
beginning below $s$, and the barcode entropy $\hbr(\varphi; L_0, L)$
is defined similarly to \eqref{eq:hbar-Reeb}. (We have suppressed
$\varphi$ in the notation for the sake of brevity.) In fact,
$\fb_\eps(s; L_0,L)$ is independent of the choice of $\hL_0$ and $\hL$
up to an additive constant, and hence $\hbr(\varphi; L_0, L)$ is
completely determined by $\varphi$ and $L_0$ and $L$. Only the
existence of exact Lagrangian fillings $\hL_0$ and $\hL$ matters.

With these definitions, the above results also hold for Reeb
flows. The Reeb analogue of \eqref{eq:hbar<htop-rel} is proved in
\cite{Fe1}. The absolute version of this inequality,
\eqref{eq:hbar<htop}, which now does not directly follow from its
relative counterpart is established in \cite{FLS}; see also
\cite{CGGM:Entropy}. The Reeb counterpart of the lower bound,
\eqref{eq:htop-hyp<hbar}, is \cite[Thm.\ B]{CGGM:Entropy}, and its
relative version, \eqref{eq:htop-hyp<hbar-rel}, is the main result of
\cite{Fe2}. Finally, we again have
$$
\hbar(\varphi)=\htop(\varphi)\textrm{ when } \dim M=3;
$$
see \cite{CGG:Entropy}. For geodesic flows, these results, other than
\eqref{eq:htop-hyp<hbar-rel}, are directly proved in \cite{GGM} in the
context of Morse theory.

\begin{Remark}[Applications]
  The lower bounds between barcode entropy and topological entropy,
  relative and absolute, for Reeb and geodesic flows along the
  lines of \eqref{eq:hbar<htop-rel} and \eqref{eq:hbar<htop} were
  established in \cite{CGGM:Entropy, FLS, Fe1, GGM}. They encompass
  several ``classical'' results relating positivity of topological
  entropy to topology of the underlying Riemannian manifold or
  symplectic topology of the contact manifold. The first of these is
  Dinaburg's theorem, \cite{Di}, giving a lower bound for the
  topological entropy of a geodesic flow via the growth of the
  fundamental group. Likewise, the lower bound from \cite{Pa} for the
  topological entropy by the exponential growth rate of the homology
  of the based loop space immediately follows from \cite{GGM}. Both of
  these results have been generalized to Reeb flows on unit
  (co)tangent bundles in \cite{MS} and \cite[Prop.\ 1.8]{FS},
  and this generalization is in turn covered by \cite{Fe1}. Several
  yet more general lower bounds on topological entropy of Reeb flows
  via the growth of various types of symplectic or contact homology
  have been proved in \cite{Al:Anosov,Al:Cyl,Al:Leg,AM,Me:Th}. Some of
  these lower bounds also follow, directly or at least conceptually,
  from the results of \cite{CGGM:Entropy,FLS,Fe1}.
\end{Remark}

\begin{Remark}
  Connections between the properties of Floer homology type
  persistence modules and symplectic dynamics go far beyond
  topological and metric entropy. We refer the reader to, e.g.,
  \cite{BG,CGG:Spectral,CGG:ReebHZ,GU,Hu,LRSV,PRSZ,PS,Sh,SZ,UZ} for an
  admittedly incomplete list of sample results and further discussion
  and references. 
\end{Remark}

\subsection{Lagrangian pseudo-chord measures}
\label{sec:Lagr}
The definition of approximate chord measures and pseudo-chord
measures, which are central to our results, can be given in a very
general context of topological dynamics.

Let $\varphi\colon M\to M$ be a compactly supported diffeomorphism of
a smooth manifold $M$.  By definition, the $k$-orbit $\CO:=\CO_k(x)$ of
$\varphi$ through $x\in M$ is a sequence
$\{x, \varphi(x), \ldots, \varphi^{k}(x)\}$. To $\CO$, we associate
the probability measure
\begin{equation}
  \label{eq:delta_O}
\delta_\CO=\frac{1}{k+1}\sum_{i=0}^k \delta_{\varphi^i(x)}.
\end{equation}
Likewise, we associate the probability measure
\begin{equation}
  \label{eq:delta_Gamma}
\delta_\Gamma=\frac{1}{m}\sum_{i=1}^m \delta_{\CO_k(x_i)}
\end{equation}
to a collection of $k$-orbits
$\Gamma=\{\CO_k(x_1),\ldots,\CO_k(x_m)\}$. The points $x_i\in M$ are
not required to be distinct, and thus $\Gamma$ is a multiset. Note
however that here the ``length'' $k$ is the same for all orbits in the
collection.

Next, let $L_0$ and $L$ be two subsets $M$. The case we are actually
interested in is where $L_0$ and $L$ are properly embedded
submanifolds of $M$. Fix an open set $U\supset L$. The reader should
think of $U$ as a small tubular neighborhood of $L$ when the latter is
a submanifold. A \emph{$U$-approximate $k$-chord} (or just an
approximate chord when $U$ and $k$ are clear from the context) is a
$k$-orbit $\CO_k(x)$ with $x\in L_0$ and $\varphi^k(x)\in U$.

Assume that for a sequence $k_j\to \infty$ and a sequence
$\Gamma_j=\{\CO_{k_j}(x_{j1}),\ldots,\CO_{k_j}(x_{{jm_j}})\}$ of
collections of $U$-approximate $k_j$-chords the measures
$\delta_{\Gamma_j}$ converge to a measure $\mu$ in the weak$^*$
topology. Then we say that $\mu$ is a \emph{$U$-approximate chord
  measure} (associated with $L_0$ and $L$). Note that $\mu$ is
automatically an invariant probability measure. The set of
  $U$-approximate chord measures is weak$^*$-closed. This follows by
  the diagonal process from the fact that the set (with the weak$^*$
  topology) of probability measures on a compact metric space is
  metrizable; \cite{Ru}.

Furthermore, an
invariant probability measure $\mu$ is said to be a \emph{pseudo-chord
  measure} (again, associated with $L_0$ and $L$) if
$\mu=\lim\mu_\ell$ for a sequence $U_\ell$, $\ell\in\N$, of open
neighborhoods of $L$ shrinking to $L$ and a sequence $\mu_\ell$ of
$U_\ell$-approximate chord measures. Here we say that $U_\ell$ is
shrinking to $L$ if
$$
\bigcap_\ell U_\ell=L
$$
and the neighborhoods $U_\ell$ are nested, i.e.,
$U_1\supset U_2\supset\ldots\supset L$.

More explicitly, $\mu$ is a pseudo-chord measure if there exists a
sequence of neighborhood $U_\ell$ shrinking to $L$ and a sequence of
collections $\Gamma_\ell$ of $U_\ell$-approximate $k_\ell$-chords with
$k_\ell\to\infty$ such that $\delta_{\Gamma_\ell}\to \mu$. Note that
here the chords comprising $\Gamma_\ell$ are not necessarily true
chords from $L_0$ to $L$ but only approximate chords. (Hence,
``pseudo'' in the term.) In general, $\mu$ cannot be obtained as the
limit of measures $\delta_\Gamma$ where $\Gamma$ comprises genuine
chords from $L_0$ to $L$; see Remark \ref{rmk:POvsPPO}. We will call
such measures \emph{genuine (or true) chord measures}.

In these definitions, we could have assumed that $M$ is just a metric
space and $\varphi$ is a continuous map. However, in such a great
generality the definition is probably meaningless. On the other hand,
as we have already mentioned, the case where $L_0$ and $L$ are
properly embedded submanifolds of $M$ is already of interest; see
Section~\ref{sec:proof-hbar-hmu}.

In this paper we are concerned with the role of pseudo-chord measures
in symplectic dynamics.  Thus we will require $L_0$ and $L$ to be
closed Lagrangian submanifolds of a symplectic manifold $M$ and
$\varphi:=\varphi_H\colon M\to M$ to be a compactly supported
Hamiltonian diffeomorphism. In this setting we say that a pseudo-chord
measure associated with $L_0$ and $L$ is \emph{Lagrangian}. For the
sake of simplicity, we will assume that $L_0$ and $L$ are monotone,
although this condition can be significantly relaxed. As in the
previous section, the manifold $M$ is required to be closed or
sufficiently nice at infinity (e.g., convex or tame) so that
Lagrangian Floer homology for closed monotone Lagrangian submanifolds
is defined.

The sets of approximate chord measures or pseudo-chord measures can be
empty. We say that $L$ is \emph{accessible} from $L_0$ if for every
neighborhood $U$ of $L$ we have $L_k\cap U\neq\emptyset$ for
infinitely many $k\in\N$, where $L_k = \varphi^k(L_0)$. This is
equivalent to that the set of pseudo-chord measures associated with
$L_0$ and $L$ is non-empty. Clearly, $L$ is accessible from $L_0$
whenever $\hbar(\varphi; L_0,L)>0$.

\begin{Theorem}
  \label{thm:main-hbar-hmu}
  Let $L_0$, $L$, $U$ and $\varphi$ be as above. Assume that $L$ is
  accessible from $L_0$. Then there exists a $U$-approximate chord
  measure $\mu$ associated with $L_0$ and $L$ such that
\begin{equation}
  \label{eq:main-chord}
\hbar(\varphi; L_0,L)\leq\hh_\mu(\varphi).
\end{equation}
\end{Theorem}

Shrinking $U$ to $L$, passing to the limit and using the fact that the
map $\mu\mapsto \hh_\mu$ is upper semi-continuous in $\mu$ (see
\cite{Ne}) we obtain the following. Recall that $\supp(\varphi)$ is
compact, and hence the set of probability measures on $\supp(\varphi)$
is compact in the weak$^*$ topology.

\begin{Corollary}
  \label{cor:main-hbar-hmu}
  There exists a pseudo-chord measure $\mu$ associated with $L_0$ and $L$
  satisfying~\eqref{eq:main-chord}.
\end{Corollary}

Since $\hh_\mu(\varphi)\leq \htop(\varphi)$ for any probability
measure $\mu$, as a consequence of Corollary \ref{cor:main-hbar-hmu}
we recover the lower bound, \eqref{eq:hbar<htop-rel}, on the
topological entropy via the barcode entropy from \cite[Thm.\
5.1]{CGG:Entropy}: $\hbar(\varphi, L_0,L)\leq\htop(\varphi)$. Also
note that the accessibility condition plays a purely formal role in
these results: without it the left-hand side is zero and the
right-hand side is not defined. It is essential for the proof of
Theorem \ref{thm:main-hbar-hmu} that the chords in question are only
approximate rather than genuine chords from $L_0$ to $L$, i.e., that
the end-points belong to $U$ but not necessarily to $L$.  We do not
know if Corollary \ref{cor:main-hbar-hmu} would still hold with the
set of pseudo-chord measures replaced by the more narrow class of
genuine chord measures; cf.\ Remark~\ref{rmk:POvsPPO}.

The proof of Theorem \ref{thm:main-hbar-hmu}, given in Section
\ref{sec:proof-hbar-hmu}, comprises two steps and utilizes yet another
entropy-type invariant -- the volume growth entropy
$\hvol(\varphi; L_0,U)$ defined in that section. The first step is the
inequality $\hbar(\varphi; L_0, L)\leq \hvol(\varphi; L_0,U)$
established in \cite{CGG:Entropy}. The second step is the inequality
$\hvol(\varphi; L_0,U)\leq \hh_\mu(\varphi)$ which holds in much greater
generality and is based on (the proof of) Yomdin's theorem; \cite{Yo}.

\begin{Remark}
  \label{rmk:chord_measures}
  The class of all Lagrangian pseudo-chord measures, where we allow
  $L_0$ and $L$ to vary, is quite large. For instance, it includes all
  ergodic measures; see Proposition \ref{prop:ergodic-apo}.  As a
  consequence, the variational principle holds for the class of all
  pseudo-chord measures, i.e.,
  $\sup_\mu\hh_\mu(\varphi)=\htop(\varphi)$ where the supremum is
  taken over all such measures $\mu$. At the same time it is not clear
  if a linear combination of pseudo-chord measures is again a
  pseudo-chord measure even with $L_0$ and $L$ fixed. (The problem is
  that two such measures can have different time sequences.)
\end{Remark}

\subsection{Approximate periodic measures}
\label{sec:Per}
Theorem \ref{thm:main-hbar-hmu} and the constructions from the
previous section readily translate to a relation between absolute
barcode entropy and metric entropy for measures associated with
approximate periodic orbits. However, it turns out that the resulting
class of measures is quite large and includes all ergodic measures
similarly to the class of all pseudo-chord measures; see Remark
\ref{rmk:chord_measures}. As a consequence, an analogue of
\eqref{eq:main-chord} in this setting follows immediately from
\cite[Thm.\ 5.1]{CGG:Entropy} and the fact that $C^\infty$-maps have
ergodic measures with maximal entropy, see \cite{Ne}. Hence, formally
speaking, the inequality does not give new information in this
case. However, we find the construction illuminating enough to warrant
a brief discussion.

Assume for the sake of simplicity that $M$ is a closed weakly monotone
symplectic manifold and $\varphi\colon M\to M$ is a Hamiltonian
diffeomorphism.  Let us apply the construction from Section
\ref{sec:Lagr} to the diagonal $\Delta=L_0=L$ in the symplectic
product $M\times M$ and the Hamiltonian diffeomorphism
$\id\times \varphi$ of $M\times M$. Pushing forward (approximate)
chord measures on $M\times M$ to the second factor results in the
notions of an approximate periodic measure and a pseudo-periodic
measure. More explicitly, the construction is as follows.

Fix $\eta>0$ and a distance function $d$ on $M$, and let
\begin{equation}
  \label{eq:U_eta}
U_\eta:=\{(x,y)\in
M\times M\mid d(x,y)<\eta\}.
\end{equation}
In the notation from Section \ref{sec:Lagr}, an
\emph{$\eta$-approximate $k$-periodic orbit} is an orbit $\CO_k(x)$
such that $d\big(x, \varphi^k(x)\big)<\eta$, i.e.,
$(x,\varphi^k(x))\in U_\eta$. To a collection $\Gamma$ of
$\eta$-approximate $k$-periodic orbits (with the same $k$), we
associate a probability measure $\delta_\Gamma$ given by
\eqref{eq:delta_O} and \eqref{eq:delta_Gamma}. An
\emph{$\eta$-approximate periodic measure} $\mu$ is the weak$^*$ limit
of a sequence of such measures $\delta_{\Gamma_j}$ as
$k_j\to\infty$. Clearly, $\mu$ is automatically an invariant
probability measure, and we could have replaced $U_\eta$ by an
arbitrary neighborhood $U$ of the diagonal. Furthermore, an invariant
probability measure $\mu$ is said to be a \emph{pseudo-periodic
  measure} if $\mu=\lim\mu_\ell$ for a sequence $\eta_\ell\to 0+$ and
a sequence of $\eta_\ell$-approximate periodic measures $\mu_\ell$.
In other words, $\mu$ is a pseudo-periodic measure if there exists a
sequence $\eta_\ell\to 0+$ and a sequence $\Gamma_\ell$ of collections
of $\eta_\ell$-approximate $k_\ell$-periodic orbits with
$k_\ell\to\infty$ such that $\delta_{\Gamma_\ell}\to \mu$.

In this setting, applying Theorem \ref{thm:main-hbar-hmu} or Corollary
\ref{cor:main-hbar-hmu} to the diagonal $L_0=L$ in $M\times M$,
replacing $\varphi$ by $\id\times \varphi$ and then projecting the
measure to the second factor, we obtain the inequality
\begin{equation}
  \label{eq:per}
\hbar(\varphi)\leq\hh_\mu(\varphi)
\end{equation}
from Corollary \ref{cor:main-hbar-hmu}. Here we use the identity
$\hh_\eta (\id \otimes \varphi) = \hh_{\pi_*\eta} (\varphi)$ which
holds for any invariant probably measure $\eta$. (This can be seen
using the Abramov--Rokhlin formula \cite{AR} and the fact that the
fiberwise entropy is zero.)  However, it is easy to show that in
general every ergodic measure is pseudo-periodic; see Proposition
\ref{prop:ergodic-apo}. Hence, as we have already pointed out,
\eqref{eq:per} also readily follows from \cite[Thm.\ 5.1]{CGG:Entropy}
and the fact that the function $\mu\mapsto\hh_\mu(\varphi)$ is upper
semi-continuous for $C^\infty$-diffeomorphisms, and hence $\varphi$
admits ergodic measures with maximal entropy, i.e.,
$\hh_\mu(\varphi)=\htop(\varphi)$; \cite{Ne}.

\begin{Remark}[Pseudo-periodic vs.\ periodic]
  \label{rmk:POvsPPO}

  A \emph{periodic measure} is defined as the limit of measures
  $\delta_\Gamma$ where all orbits comprising $\Gamma$ are periodic,
  but possibly have different periods; see, e.g., \cite{Ma:PM}. On the other hand, when
  we apply the diagonal construction to genuine chord measures, we
  obtain the limits of measures $\delta_{\Gamma_\ell}$ where all
  orbits in $\Gamma_\ell$ have the same period. By analogy with the
  chord case, we would call such measures \emph{genuinely
    periodic}. However, for discrete time the two notions agree as
  one can see from the proof of Proposition \ref{prop:pseudo-per}.

Furthermore, a periodic measure is pseudo-periodic in both
  continuous and discrete cases; see Proposition
  \ref{prop:pseudo-per}. The converse is not true: there are
pseudo-periodic measures which are not periodic. For instance, let
$\varphi\colon S^2\to S^2$ be an Anosov--Katok pseudo-rotation;
\cite{AK}. This is an area-preserving diffeomorphism of $S^2$ with
exactly three ergodic measures: two fixed points and the area form
$\omega$. Then $\omega$ is pseudo-periodic by Proposition
\ref{prop:ergodic-apo}. On the other hand, a periodic measure is
necessarily a convex linear combination of the $\delta$-measures at
the fixed points. In particular, a pseudo-chord measure from $L_0$ to
$L$ need not be genuine chord measure, i.e., the limit of measures
associated with true chords.  Similar examples exist for broad classes
of Reeb flows and Hamiltonian diffeomorphisms in all dimensions;
\cite{Ka,LRS}.
\end{Remark}

We do not know exactly how large the class of pseudo-periodic measures
is. While genuine periodic and ergodic measures are pseudo-periodic,
it is not clear if a convex linear combination of pseudo-periodic
measures is pseudo-periodic; cf.\ Remark~\ref{rmk:chord_measures}.

\subsection{Reeb flows}
\label{sec:Reeb}
The definitions and results from the previous two sections extend to
Reeb flows in a straightforward way and here we only briefly comment
on them.

Let $\varphi^t$ be a flow on a closed manifold $M$. To an integral
curve $\CO:=\CO_T(x)\colon [0,T]\to M$, $t\mapsto \varphi^t(x)$, we
associate the measure $\delta_\CO$ defined by
$$
\delta_\CO(f):=\frac{1}{T}\int_0^T f\big(\varphi^t(x)\big)\,dt,
$$
where we view measures as elements of the dual space to $C^0(M)$. For
a collection $\Gamma$ of $m$ not necessarily distinct orbits
$\CO_i:=\CO_T(x_i)$, we set
$$
\delta_\Gamma:=\frac{1}{m}\sum \delta_{\CO_i}.
$$
In what follows, we will always requite the time $T$ to be the same
for all orbits in $\Gamma$.

Next, let $L_0$ and $L$ be two closed subsets (e.g., closed
submanifolds) of $M$. Pick an open neighborhood $U\supset L$.  A
\emph{$U$-approximate $T$-chord} from $L_0$ to $L$ is an integral
curve $\CO_T(x)$ connecting $x\in L_0$ to $\varphi^T(x)\in U$. We say
that $\mu$ is a \emph{$U$-approximate chord measure} associated with
$L_0$ and $L$ if there exists a sequence of times $T_j\to \infty$ and
collections $\Gamma_j$ of $U$-approximate $T_j$-chords such that
$\delta_{\Gamma_j}\to \mu$ in the weak$^*$ topology. A pseudo-chord measure
associated with $L_0$ and $L$ is the limit of $U_\ell$-approximate
chord measures where neighborhoods $U_\ell$ shrink to $L$.

Let now $\varphi^t$ be the Reeb flow on the boundary $M=\p W$ of a
Liouville domain $W$. Furthermore, as in Section \ref{sec:barcode},
let $L_0$ and $L$ be Legendrian submanifolds of $M$ admitting exact
Lagrangian fillings $\hL_0$ and $\hL$ to $W$ which we can assume to
be asymptotically conical. Similarly to the Hamiltonian case, we say
that $L$ is \emph{accessible} from $L_0$ if for every neighborhood $U$
of $L$ we have $\varphi^{T_k}(L_0)\cap U\neq\emptyset$ for some sequence
$T_k\to\infty$. This is equivalent to that the set of pseudo-chord measures
is non-empty, which is the case if (but not only if)
$\hbar(\varphi,L_0,L)>0$.

\begin{Theorem}
  \label{thm:Reeb-hbar-chord}
  Assume that $L$ is accessible from $L_0$. Then for every
  $U\supset L$ there exists a $U$-approximate chord measure associated
  with $L_0$ and $L$ and hence also a pseudo-chord measure $\mu$ such
  that
$$
\hbar(\varphi; L_0,L)\leq\hh_\mu(\varphi).
$$
\end{Theorem}

This is a version of Theorem \ref{thm:main-hbar-hmu} and Corollary
\ref{cor:main-hbar-hmu} for Reeb flows. Remark
\ref{rmk:chord_measures} carries over to Legendrian chord measures
word-for-word. Again the accessibility requirement plays a purely
formal role: without it the left-hand side is zero and the right-hand
side is not defined. As in the Hamiltonian case, we do not know if
Theorem \ref{thm:Reeb-hbar-chord} holds for genuine chord
measures. The proof of Theorem \ref{thm:Reeb-hbar-chord} is quite
similar to the proof of Theorem \ref{thm:main-hbar-hmu} and is given
in Section \ref{sec:proof-hbar-hmu-Reeb}.

One can define approximate periodic measures and pseudo-periodic 
measures for flows by changing in a similar fashion discrete time to
continuous time in the definitions from Section \ref{sec:Per}. Then
again we have an analogue of Theorem \ref{thm:Reeb-hbar-chord},
similar to \eqref{eq:per}, for such measures, but it again follows
from known results.

\section{Proofs and refinements}

\subsection{Proof of Theorem \ref{thm:main-hbar-hmu}, refinements and
  further comments}
\label{sec:proof-hbar-hmu}
We will prove a slightly more precise version of Theorem
\ref{thm:main-hbar-hmu}. To state this result, we first need to make
several general definitions.  As in Section \ref{sec:Lagr}, let $L_0$
and $L_1$ be two closed submanifolds of $M$ and $\varphi\colon M\to M$
be a compactly supported diffeomorphism.

Fix an arbitrary Riemannian metric on $M$. Let $U$ be a fixed
neighborhood of $L$ and $L_k:=\varphi^k(L_0)$. We define the
\emph{volume growth entropy} for the triple $(\varphi; L_0, U)$ as
\begin{equation}
  \label{eq:hvol1}
  \hvol(\varphi; L_0, U) :=\limsup_{k\to\infty}
  \frac{\log^+\vol(L_{k}\cap U)}{k}.
\end{equation}
Clearly, $\hvol(\varphi; L_0, U)$ is independent of the background
metric. Furthermore, $\hvol$ is monotone increasing in $U$, i.e.,
$\hvol(\varphi; L_0, U')\leq \hvol(\varphi, L_0, U)$ when
$U'\subset U$. (As a side note, we see no reason why $\hvol$ would
have any semi-continuity properties in $U$.) Set
\begin{equation}
  \label{eq:hvol2}
\hvol(\varphi; L_0, L) :=\inf_{U\supset L}\hvol(\varphi; L_0, U).
\end{equation}
Note that here the infimum can be replaced by the lower limit.

This definition readily translates to the framework of approximate
periodic orbits in Section \ref{sec:Per} via the graph
construction. Even though we are only tangentially interested in this
case, let us briefly spell it out.  Let $U$ be a neighborhood of the
diagonal $\Delta$ in $M\times M$ and $G_k$ be the graph of
$\varphi^k$. We set
$$
\hvol(\varphi; U) :=\limsup_{k\to\infty}\frac{\log^+\vol(G_{k}\cap
  U)}{k}
\textrm{ and }
  \hvol(\varphi) :=\inf_{U\supset \Delta}\hvol(\varphi; U).
$$
Clearly, $\hvol(\varphi; U)$ and $\hvol(\varphi)$ are independent of
the background metric, and $\hvol(\varphi; U)$ is monotone in $U$.
The infimum can be replaced by the lower limit.

\begin{Theorem}
  \label{thm:vol-chord}
  In the setting of Theorem \ref{thm:main-hbar-hmu}, there exists a
  $U$-approximate chord measure $\mu$ associated with $L_0$ and $L$
  such that
  \begin{equation}
    \label{eq:vol-chord}
\hbar(\varphi; L_0,L)\leq\hvol(\varphi; L_0,U)\leq \hh_\mu(\varphi)
\end{equation}
and, as a consequence, a pseudo-chord measure $\mu$, again associated
with $L_0$ and $L$, with
$$
\hbar(\varphi; L_0,L)\leq\hvol(\varphi; L_0,L)\leq \hh_\mu(\varphi).
$$
\end{Theorem}

Theorem \ref{thm:main-hbar-hmu} and hence Corollary
\ref{cor:main-hbar-hmu} immediately follow from this result by
skipping the middle terms in the inequalities. An analogue of Theorem
\ref{thm:vol-chord} holds for a pseudo-periodic measure $\mu$, i.e.,
$\hbar(\varphi)\leq \hvol(\varphi)\leq \hh_\mu(\varphi)$. However, as
we have already pointed out, the first inequality is proved in
\cite{CGG:Entropy} and the second one does not carry new information.

\begin{proof}[Proof of Theorem \ref{thm:vol-chord}]
  Throughout the proof we can assume that $\hbar(\varphi, L_0,L)>0$;
  for otherwise the statement is void. The argument comprises two
  steps establishing, respectively, the first and the second
  inequality in \eqref{eq:vol-chord}.

  \medskip\noindent\emph{Step 1: From barcode entropy to volume
    growth.}  Let $L_0$ and $L$ be Hamiltonian isotopic, closed
  monotone Lagrangian submanifolds and $\varphi$ be a Hamiltonian
  diffeomorphism as in the statement of the theorem. Then, by
  \cite[Sec\ 5.2.3 and Cor.\ 5.7]{CGG:Entropy},
\begin{equation}
\label{eq:vol-growth2}
\hbar(\varphi; L_0, L)\leq \hvol(\varphi; L_0, U) 
\end{equation}
for every open set $U\supset L$. This proves the first part of
\eqref{eq:vol-chord}.

\medskip\noindent\emph{Step 2: From volume growth to separated chords
  to metric entropy.} Our next goal is to prove the second inequality
in \eqref{eq:vol-chord}. This step of the proof is completely general,
and it is sufficient to assume that $L_0$ and $L_1$ are two closed
submanifolds of $M$ and $\varphi\colon M\to M$ is a compactly supported
$C^\infty$-smooth diffeomorphism. The $C^\infty$-smoothness
requirement is essential.

\begin{Proposition}
  \label{prop:vol-chord}
  There exists a $U$-approximate chord measure $\mu$ associated
    with $L_0$ and $L$ such that
  $$
  \hvol(\varphi; L_0, U)\leq \hh_\mu(\varphi).
  $$
  As a consequence,
  $$
  \hvol(\varphi; L_0, L)\leq \hh_\mu(\varphi)
  $$
  for some pseudo-chord measure $\mu$ associated with $L_0$ and $L$.
\end{Proposition}

Theorem \ref{thm:vol-chord} immediately follows from
\eqref{eq:vol-growth2}, establishing the first inequality in
\eqref{eq:vol-chord}, and Proposition \ref{prop:vol-chord} proving the
second.

\begin{proof}
Set
$$
d_k(x,y):=\max\{ d\big(\varphi^i(x), \varphi^i(x)\big)\mid 0 \leq i\leq
k\}.
$$
We say that two $k$-orbits
$$
\CO_k(x):=\{x,\varphi(x),\ldots,\varphi^k(x)\}
\textrm{ and }
\CO_k(y):=\{y,\varphi(y),\ldots,\varphi^k(y)\}
$$
are $\eta$-$d_k$-separated if $d_k(x,y)>\eta$ or, in other words,
$d\big(\varphi^i(x), \varphi^i(y)\big)>\eta$ for some $0\leq i\leq
k$. Let $\CS_k(\eta)$ be a maximal collection of
$\eta$-$d_k$-separated points $x\in L_0$ with $\varphi^k(x)\in
U$. This set is naturally in one-to-one correspondence with the set of
$\eta$-$d_k$-separated chords from $L_0$ to $U$ where we identify an
approximate chord with its initial condition. Note that while this
collection might not be unique, the cardinality $|\CS_k(\eta)|$ is
well defined. The key to the proof is the following lemma, which is
ultimately a consequence of the ``local Yomdin theorem''; see
\cite[Thm.\ 1.8]{Yo}.

\begin{Lemma}
  \label{lemma:separated}
  Pick any constant $0<a<\hvol(\varphi,L_0,U)$. Then, when $\eta>0$ is
  sufficiently small,
  $$
  a \leq \limsup_{k\to\infty}\frac{\log^+|\CS_{k}(\eta)|}{k}.
  $$
\end{Lemma}

\begin{proof}
  Denote by $B_k(x,\eta)$ the $d_k$-ball of radius $\eta$ centered at
  $x$. Choose a maximal collection $S_k(\eta)$ and fix $\eta'>\eta$,
  e.g., $\eta'=2\eta$. Then
$$
L_k\cap U\subset \bigcup_{x\in \CS_k(\eta)}
\varphi^k\big(B_k(x,\eta')\big).
$$
Thus
\begin{equation}
  \label{eq:vol-upper-bound}
\vol(L_k\cap U)\leq |\CS_k(\eta)|\cdot\max_{x\in
  \CS_k(\eta)}\vol\big(\varphi^k(L_0\cap B_k(x,\eta'))\big).
\end{equation}
Set
$$
E_k(x,\eta'):=\frac{\log^+ \vol\big(\varphi^k(L_0\cap
  B_k(x,\eta'))\big)}{k}.
$$
Then, by \eqref{eq:vol-upper-bound},
$$
\frac{\log^+ \vol(L_k\cap U)}{k}\leq \frac{\log^+ |\CS_k(\eta)|}{k}+ \max_{x\in
  \CS_k(\eta)} E_k(x,\eta').
$$
By \cite[Thm.\ 1.8]{Yo},
\begin{equation}
  \label{eq:Yomdin}
\lim_{\eta'\to 0+}\limsup_{k\to\infty}\sup_{x\in L_0} E_k(x,\eta')=0
\end{equation}
whenever $\varphi$ is $C^\infty$-smooth.
The lemma follows from \eqref{eq:Yomdin}.
\end{proof}

In the setting of Lemma \ref{lemma:separated}, pick a sequence
$k_j\to \infty$ such that
\begin{equation}
  \label{eq:vol-growth}
a \leq \lim_{k_j\to\infty}\frac{\log^+|\CS_{k_j}(\eta)|}{k_j}.
\end{equation}
Let $\Gamma_{k_j}$ be the set of $U$-approximate $k_j$-chords with
initial conditions $\CS_{k_j}(\eta)$. After passing to a subsequence,
we can assume that $\delta_{k_j}\to \mu$ (in the weak$^*$ topology)
for some $U$-approximate chord measure $\mu$.  Clearly,
$|\Gamma_{k_j}|=|\CS_{k_j}(\eta)|$. Since $\CS_k(\eta)$ is an
$\eta$-$d_k$-separated set, by \eqref{eq:vol-growth} and \cite[Lemma
4.5.2]{KH},
$$
a\leq \lim_{k_j\to \infty} \frac{\log^+ |\CS_{k_j}(\eta)|}{k_j}\leq
\hh_\mu(\varphi).
$$
Finally, by sending $a\to \hvol(\varphi; L_0,U)$ and passing to
  the limit again, we obtain a $U$-approximate chord measure $\mu$
  such that $\hvol(\varphi; L_0,U)\leq \hh_\mu(\varphi)$.

\end{proof}
This completes the proof of Theorem \ref{thm:vol-chord}.
\end{proof}

In the rest of this section we will elaborate on the notions of an
approximate periodic measure and a pseudo-periodic measure from
Section \ref{sec:Per}. Below we make no background symplectic
  topological assumptions on symplectic or Lagrangian manifolds. For
  instance, in Proposition \ref{prop:ergodic-apo} we do not claim
  that every symplectic manifold admits a closed monotone Lagrangian
  submanifold.

\begin{Proposition}
  \label{prop:ergodic-apo}
For any diffeomorphism  of a closed manifold or a flow every ergodic
measure is pseudo-periodic. Likewise, every ergodic measure is a Lagrangian
or Legendrian pseudo-chord measure for Hamiltonian diffeomorphisms or
Reeb flows for a suitable choice of Lagrangian or Legendrian submanifolds.
\end{Proposition}  

\begin{proof}
  We will focus on the case of diffeomorphisms; for flows the argument
  is similar. Let $\mu$ be an ergodic measure. By Poincar\'e's Recurrence
  Theorem, $\mu$-a.e.\ point $x\in M$ is recurrent, i.e., for
  $\mu$-a.a.\ $x\in M$ the trajectory $\{\varphi^k(x)\mid k\in \N\}$
  gets arbitrarily close to $x$. Furthermore, for $\mu$-a.a.\
  $x\in M$ we have
$$
\mu=\lim_{k\to\infty}\frac{1}{k+1}\sum_{i=0}^k \delta_{\varphi^i(x)}
$$
due to Birkhoff's Ergodic Theorem. Taking $x$ in the intersection of
these two full measure sets, we see that $\mu$ is an
$\eta$-approximate periodic measure for any $\eta>0$. As a
consequence, $\mu$ is a pseudo-periodic measure.

For chord measures it now suffices to take any Lagrangian (Legendrian)
submanifolds $L_0$ and $L$ intersecting at $x$.
\end{proof}

Revisiting Remark \ref{rmk:POvsPPO}, recall that the limit of measures
$\delta_\Gamma$ is a \emph{periodic measure} when all orbits
comprising $\Gamma$ are periodic, but possibly have different
periods. Such measures are automatically pseudo-periodic as the
following simple observation shows.

\begin{Proposition}
\label{prop:pseudo-per}  
For any diffeomorphism of a closed manifold or a flow every periodic
measure is pseudo-periodic.
\end{Proposition}

\begin{proof}
  Let $\mu$ be a periodic measure. We start with the discrete case. By
  definition, there exists a sequence $\Gamma_j$ such that
  $\delta_{\Gamma_j}\to \mu$ and each $\Gamma_j$ is a finite
  collection of periodic orbits $\CO_{k_{ij}}(x_{ij})$. Let $k_j$ be
  the least common multiple of all periods $k_{ij}$ and $\Gamma'_j$
  comprise the periodic orbits $\CO_{k_j}(x_{ij})$ of the same period
  $k_j$. Then $\Gamma'_j$ is a collection of $\eta$-approximate
  periodic orbits in the sense of Section \ref{sec:Per} for any
  $\eta>0$, and $\delta_{\Gamma'_j}=\delta_{\Gamma_j}$.  Hence,
    $\mu =\lim \delta_{\Gamma'_j}$ is a genuine periodic measure in
    the sense of Remark \ref{rmk:POvsPPO} and, in particular, a
    pseudo-periodic measure.

  For continuous time the argument is slightly different with the
  exact equality of measures and periods replaced by
  approximations. Now we have $\delta_{\Gamma_j}\to \mu$ where each
  $\Gamma_j$ is a finite collection of periodic orbits
  $\CO_{T_{ij}}(x_{ij})$. For every $\eps>0$, there exists an
  arbitrarily large $T_j$ such that $|T_j-k_{ij}T_{ij}|<\eps$ for some
  integers $k_{ij}>0$. Let $\Gamma'_j$ be the multiset of the orbits
  $\CO_{T_{j}}(x_{ij})$. Note that we can make the norm of the
  difference $\delta_{\Gamma_j}-\delta_{\Gamma'_j}$ in the dual of
  $C^0(M)$ arbitrarily small by making $\eps>0$ small. Likewise, given
  $\eta>0$, the orbits $\CO_{T_{j}}(x_{ij})$ are $\eta$-approximately
  periodic when $\eps$ is small. Thus for any sequence
  $\eta_\ell\to 0+$ we can find a sequence $\Gamma'_\ell$ of
  collections of $\eta_\ell$-approximately periodic orbits such that
  we still have $\mu=\lim \delta_{\Gamma'_\ell}$.
\end{proof}

\subsection{On the proof of Theorem \ref{thm:Reeb-hbar-chord}}
\label{sec:proof-hbar-hmu-Reeb}
For continuous time too, it is convenient to prove a slightly more
precise version of the theorem by inserting the volume growth entropy
as an intermediate term.  The argument follows the same line of
reasoning as the proof in the previous section of Theorem
\ref{thm:main-hbar-hmu} and, moreover, the theorem can be obtained as
a consequence of that proof.

Let us start with preliminary definitions. In the setting of Section
\ref{sec:Reeb}, fix a neighborhood $U$ of $L$ and set
$$
V(s, U):=\vol(L_s\cap U),
$$
where $L_s:=\varphi^s(L_0)$.
We define the volume growth entropy of
a flow $\varphi$ by
\begin{equation}
  \label{eq:vol1-flow}
\hvol(\varphi; L_0, U) :=\limsup_{s\to\infty}\frac{\log^+ V(s,U)}{s}
\end{equation}
and
\begin{equation}
  \label{eq:vol2-flow}
  \hvol(\varphi; L_0,L)
  :=\lim_{U \to L}\hvol(\varphi; L_0, U)
  =\inf_{U\supset L}  \hvol(\varphi; L_0, U).
\end{equation}

Our next goal is to connect this definition with the definition of the
volume growth entropy for diffeomorphisms from the previous section. To
this end, we first note that while up to this point the
difference between maps and flows has always been clear from the
context, it is not necessarily so below. Hence, in what follows
$\varphi$ refers to the flow while its time-$T$ map will be denoted
by $\varphi^T$.

It is easy to see that for any $T>0$ and any $U\supset L$, we have
\begin{equation}
  \label{eq: vol-flow-map1}
T^{-1}\hvol(\varphi^T;L_0, U)\leq \hvol(\varphi;L_0, U).
\end{equation}
Here the volume growth entropy of the flow $\varphi$ is understood as
in \eqref{eq:vol1-flow} and \eqref{eq:vol2-flow}, while the volume
growth entropy of the time-$T$ map $\varphi^T$ is defined in
\eqref{eq:hvol1} and \eqref{eq:hvol2}.  The opposite inequality, given
by the following lemma, is more subtle.

\begin{Lemma}
  \label{lemma:vol-flow-map}
  Assume that the closure of an open set $U'\supset L$ is contained in
  $U$. Then, whenever $T>0$ is sufficiently small (depending on $U'$
  and $U$), we have
  \begin{equation}
  \label{eq:vol-flow-map2}
  \hvol(\varphi; L_0, U')\leq T^{-1}\hvol(\varphi^T; L_0, U).
\end{equation}
\end{Lemma}

\begin{Remark}
  In Lemma \ref{lemma:vol-flow-map}, the smaller the gap between $U'$
  and $U$, the smaller the required value of $T$ is. We do not know if
  $T$ can be chosen independent of the gap and if we can have
  $\hvol(\varphi; L_0, U)= T^{-1}\hvol(\varphi^T; L_0, U)$ and
  $\hvol(\varphi; L_0, L)= T^{-1}\hvol(\varphi^T; L_0, L)$ for some
  specific value of $T$. The underlying reason is the lack
  of any clear-cut continuity of $\hvol$ with respect to $U$.
\end{Remark}  

\begin{proof}
  Let $T>0$ be so small that $\varphi^\tau(U')\subset U$ whenever
  $|\tau|\leq T$. Pick a sequence $s_j\to \infty$ such that
$$
\lim_{j\to\infty}\frac{\log^+ V(s_j,U')}{s_j}=\hvol(\varphi;U').
$$
Then for some $k_j\in\N$, we have $s_j=k_j T+\tau_j$ where
$|\tau_j|\leq T$. Hence,
$$
V(s_j,U')\leq C\cdot V(k_jT,U)
$$
for some constant $C>0$ independent of $s_j$, and
\eqref{eq:vol-flow-map2} follows.
\end{proof}

Let now $M=\p W$ be the boundary of a Liouville domain $W$ and
$\varphi^t$ the Reeb flow on $M$.  Furthermore, as in Section
\ref{sec:barcode}, let $L_0$ and $L$ be closed Legendrian submanifolds
with exact compact Lagrangian fillings to $W$. Fix an open set
$U\supset L$ in $M$.

\begin{Theorem}
  \label{thm:Reeb-hbar-chord2}
  In the setting of Theorem \ref{thm:Reeb-hbar-chord}, for every
  neighborhood $U'\supset L$ whose closure is contained in $U$, there
  exists a $U$-approximate chord measure $\mu$ associated with $L_0$
  and $L$ such that
  \begin{equation}
    \label{eq:Reeb-hbar-chord2}
\hbar(\varphi; L_0,L)\leq \hvol(\varphi;L_0,U')\leq\hh_\mu(\varphi).
\end{equation}
Hence, there is also a pseudo-chord measure $\mu$, again associated with
$L_0$ and $L$, with
$$
\hbar(\varphi; L_0,L)\leq \hvol(\varphi;L_0,L)\leq\hh_\mu(\varphi).
$$
\end{Theorem}

\begin{proof} The first inequality in \eqref{eq:Reeb-hbar-chord2} is
  an analogue of \eqref{eq:vol-growth2} for Reeb flows; it is
  established in \cite{Fe1}.  Hence we only need to prove the second
  inequality.

  Assume now that the closure of a neighborhood $U'\supset L$ is
  contained in $U$. Then, by Lemma \ref{lemma:vol-flow-map},
 $$
 \hvol(\varphi; L_0, U')\leq T^{-1}\hvol(\varphi^T; L_0, U)
 $$
 for all sufficiently small $T>0$. By Proposition \ref{prop:vol-chord}
 applied to $\varphi^T$ we have
  $$
  \hvol(\varphi^T; L_0, U)\leq \hh_\mu(\varphi^T)
  $$
  for some $U$-approximate $\varphi^T$-chord measure $\mu$. It is easy
  to see that the
  measure $\mu$ is automatically invariant under the flow $\varphi$
  and also, by construction, a $U$-approximate chord measure for the
  flow. Furthermore,
  $$
  \hh_\mu(\varphi^T)=T\hh_\mu(\varphi^1)=T\hh_\mu(\varphi).
  $$
  Combining the last two displayed formulas, we obtain the second
  inequality in \eqref{eq:Reeb-hbar-chord2}.
\end{proof}

Likewise, in the case of approximate periodic orbits, we fix a
neighborhood $U$ of the diagonal $\Delta$ in $M\times M$. For
instance, we can take $U=U_\eta$ to be the $\eta$-neighborhood of
$\Delta$ as in \eqref{eq:U_eta}. Denote by $G_s$ the graph of
$\varphi^s$ in $M\times M$ and set
$$
V(s,U):=\vol(G_s\cap U_\eta)
$$
Then the volume growth entropy, $\hvol(\varphi; U)$ and
$\hvol(\varphi)$, of the flow $\varphi$ is still defined by
\eqref{eq:vol1-flow} and \eqref{eq:vol2-flow}. The upper bound
\eqref{eq:vol-flow-map2} from Lemma \ref{lemma:vol-flow-map} takes
the form $\hvol(\varphi; U')\leq T^{-1}\hvol(\varphi^T; U)$.  Again we
have an analogue of Theorem \ref{thm:Reeb-hbar-chord2}, but it does
not contain new information as stated because every ergodic measure is
pseudo-periodic.  The first inequality in this version of the theorem
follows from the main result of \cite{FLS} (see also \cite
{CGGM:Entropy}) and the second one from the existence of maximal
entropy ergodic measures; \cite{Ne}.

\section{Examples}
\label{sec:examples}

In this section we discuss several examples further illustrating the
notion of a pseudo-chord measure.

\begin{Example}[Hyperbolic sets]
  \label{ex:Hyperbolic}
  Let $K$ be a locally maximal, topologically transitive compact
  hyperbolic set of a Hamiltonian diffeomorphism
  $\varphi\colon M\to M$. Assume that $L_0$ and $L$ intersect $K$ and
  that $L_0$ contains a ball centered at $p\in L_0\cap K$ in the
  unstable manifold of $K$ and $L$ contains a ball centered at
  $q\in L\cap K$ in the stable manifold.  Then, as shown in \cite{Me},
  $\htop(\varphi|_K)\leq \hbar(\varphi;L_0,L)$. Thus, we have
  \begin{equation}
    \label{eq:hyperbolic}
  \htop(\varphi|_K)\leq\hbar(\varphi;L_0,L)\leq \hh_\mu(\varphi)\leq
  \htop(\varphi),
\end{equation}
where all inequalities can be strict. A similar chain of inequalities
holds for Legendrian submanifolds and hyperbolic sets of Reeb flows
with the first inequality established in \cite{Fe2}.
\end{Example}

The next two examples concern with the cases where $\varphi$ is
globally hyperbolic and the pseudo-chord measure can be explicitly
identified as the maximal entropy measure.

\begin{Example}[Hyperbolic symplectomorphisms]
  \label{ex:Hyperbolic-tori}
  Let $\varphi\colon \T^{2}\to \T^{2}$ be a hyperbolic
  symplectomorphism. For instance, $\varphi$ itself can be a
  hyperbolic linear automorphism. Furthermore, let $L_0$ and $L$ be two
  embedded contractible circles in $\T^2$ bounding the same area. Thus
  $L_0$ and $L$ are Hamiltonian isotopic monotone Lagrangian
  submanifolds of $\T^{2}$.  Then $L_k:=\varphi^k(L_0)$ is still
  Hamiltonian isotopic to $L$. Hence the filtered Floer homology of
  the pair $(L,L_k)$, and hence the barcode, are defined. The results
  and constructions from \cite{CGG:Entropy,Me} and Section
  \ref{sec:Lagr} carry over to this case.

  Assume, in addition, that $L_0$ contains an arc in an unstable
  leaf of $\varphi$ and $L$ contains an arc in a stable
  leaf. Note that while stable/unstable foliations are only
  $C^0$-smooth, individual leafs are $C^\infty$. Then the results of
  \cite{Me} still apply with $K=\T^{2}$, and the inequalities in
  \eqref{eq:hyperbolic} turn into the equality
   \begin{equation}
    \label{eq:hyperbolic2}
 \hbar(\varphi;L_0,L)=\hh_\mu(\varphi)=\htop(\varphi).
\end{equation}
Hence, in this case, the pseudo-chord measure $\mu$ is the unique
maximal entropy measure, the Bowen measure; \cite[Sec.\ 20.1]{KH}. If
$\varphi$ is linear, $\mu$ is the normalized Lebesgue measure.  We do
not know what the measure $\mu$ is when the stable/unstable leaf
condition is dropped. (Note however that $L_0$ and $L$ are
automatically tangent to the unstable and stable foliations.)

We expect this example to generalize to hyperbolic symplectomorphisms
$\varphi\colon \T^{2n}\to \T^{2n}$ and Hamiltonian isotopic monotone
Lagrangian submanifolds $L_0$ and $L$ of $\T^{2n}$.  However, it is
not clear how to guarantee that then $L_k$ and $L$ are still
Hamiltonian isotopic even when $\varphi$ Hamiltonian isotopic to a
linear automorphism. Furthermore, when $2n\geq 4$, the tangency
condition is not generically satisfied for dimension reasons.
\end{Example}

\begin{Example}[Geodesic flows without conjugate points]
  \label{ex:Hyperbolic-geodesic}
  Let $Q$ be a closed manifold equipped with a metric $g$ without
  conjugate points and $\varphi$ be the geodesic flow on the unit
  (co)tangent bundle $M:=ST^*Q$. For instance, $g$ can be a metric of
  negative sectional curvature and then $\varphi$ is a hyperbolic (aka
  Anosov) Reeb flow.

  As $L_0$ and $L$ we take the unit spheres $ST^*_pM$ and $ST^*_qM$
  for two not necessarily distinct points $p$ and $q$ in $M$. These
  Legendrian submanifolds are nowhere tangent to the stable/unstable
  foliations of $\varphi$. As a consequence, even when $\varphi$ is
  hyperbolic, $L_0$ and $L_1$ do not meet the stable/unstable manifold
  condition from \cite{Fe2} alluded to in Example \ref{ex:Hyperbolic}
  for the hyperbolic set $M$. Thus we cannot use that criterion to
  show that
  $$
  \hbar(\varphi;L_0,L)=\htop(\varphi).
  $$
  However, this equality still holds.

  To see this, recall that the wrapped Floer homology persistence
  module for the pair $(L_0,L)$ is isomorphic to the Morse homology
  persistence module for the space of paths connecting $p$ and $q$
  filtered by length; \cite{AS}. (In fact, in this setting we could
  have directly used the Morse-theoretic version of barcode entropy
  from \cite{GGM}.) The latter Morse complex does not have finite bars
  and $\fb_\eps(s;L_0,L)=\fb_\infty(s;L_0,L)$ is simply the number of
  geodesics of length below $s$ connecting $p$ and $q$. By \cite{Ma},
  the exponential growth rate of this number as a function of $s$ is
  equal to $\htop(\varphi)$. (The condition that $g$ has no conjugate
  points is sufficient here.)

  Therefore, the analogue of \eqref{eq:hyperbolic2} holds in this case
  for some pseudo-chord measure $\mu$, which is then a measure of maximal
  entropy. The measure is unique when $\varphi$ is hyperbolic, and
  this is the Bowen--Margulis measure; \cite[Sec.\ 20.5]{KH}. These
  observations also carry over to the case of pseudo-periodic measures,
  but then the result is entirely expected.
\end{Example}

\end{document}